\documentclass[11pt]{article}
\usepackage{amsmath,mathrsfs}
\usepackage{amssymb}
\usepackage[title]{appendix}
\usepackage{amsthm}
\usepackage{indentfirst}
\usepackage{color}
\allowdisplaybreaks

\newtheorem{theorem}{Theorem}[section]
\newtheorem*{theorem*}{Theorem}
\newtheorem{assumption}{Assumption}
\newtheorem{corollary}[theorem]{Corollary}
\newtheorem{remark}[theorem]{Remark}
\newtheorem{lemma}[theorem]{Lemma}
\newtheorem{definition}[theorem]{Definition}

\newtheorem{proposition}[theorem]{Proposition}
\newtheorem{example}[theorem]{Example}
\newtheorem*{proposition*}{Proposition}

\newcommand{\R}{\mathbb{R}}

\newcommand{\be}{\begin{eqnarray*}}
\newcommand{\ee}{\end{eqnarray*}}
\newcommand{\ba}{\begin{align*}}
\newcommand{\bpm}{\begin{pmatrix}}
\newcommand{\epm}{\end{pmatrix}}

\begin{document}
\title{Monogenic functions over real alternative $\ast$-algebras: fundamental results and applications}

\author
{Qinghai Huo$^1$\thanks{This work was partially supported by    NSFC (No. 12301097), the Fundamental Research Funds for the Central Universities (No. JZ2025HGTB0171) and  China Scholarship Council (No. 202506690055).}
,\ Guangbin Ren$^2$\thanks{This work was partially supported by NSFC (No. 12171448).}
,\ and Zhenghua Xu$^1$\thanks{This work was partially supported by  the Anhui Provincial Natural Science Foundation (No. 2308085MA04),
 the Fundamental Research Funds for the Central Universities (No. JZ2025HGTG0250) and China Scholarship Council (No. 202506690052).}
\\
\\
\emph{$^1$\small School of Mathematics, Hefei University of Technology,}\\
\emph{\small  Hefei, 230601, P. R. China}\\
\emph{\small E-mail address: hqh86@mail.ustc.edu.cn;}  \\
\emph{\small zhxu@hfut.edu.cn}
\\
\emph{\small $^2$School of Mathematical Sciences, University of Science and Technology of China, }\\
\emph{\small Hefei 230026, P. R. China} \\
\emph{\small E-mail address:   rengb@ustc.edu.cn}
}


\maketitle

\begin{abstract}
 The concept of   monogenic  functions over real alternative $\ast$-algebras has recently  been introduced to unify several classical monogenic (or regular) functions theories in hypercomplex analysis, including quaternionic, octonionic, and Clifford analysis as special cases. This paper explores the fundamental properties of these monogenic functions in hypercomplex subspaces, focusing on  the Cauchy-Pompeiu  integral formula and Taylor series expansion,  among which the  non-commutativity and  especially  non-associativity of multiplications demand   full considerations in terms of some new techniques.
The theory presented herein provides a robust framework for understanding monogenic functions in the context of real alternative $\ast$-algebras, shedding light on the interplay between algebraic structures and hypercomplex analysis.

 \end{abstract}
\section{Introduction}
The theory of \textit{monogenic functions} has its roots in the 1930s, with seminal contributions by Moisil \cite{Moisil} and Fueter \cite{Fueter}, who introduced \textit{regular functions} of a quaternionic variable as a part of the broader study of hypercomplex analysis. This theory, initially focused on quaternions, was later extended to higher dimensions, ultimately leading to the development of monogenic functions over Clifford algebras. There are extensive  references, far from completeness, we refer to  \cite{Brackx,Colombo-Sabadini-Sommen-Struppa-04,Delanghe-Sommen-Soucek-92,Gurlebeck}.  A distinguishing feature of all Clifford algebras is their associativity, which greatly simplifies the analysis of the corresponding function theories. In 1973, Dentoni and Sce \cite{Dentoni-Sce}  (see \cite{Colombo-Sabadini-Struppa-20} for an English translation) extended the notion of regular functions from quaternions to the non-associative algebra of octonions (also called Cayley numbers), marking a pivotal development in the theory of functions over hypercomplex numbers.
See for instance \cite{Colombo-Sabadini-Struppa-00,Li-Peng-01,Li-Peng-02,Li-Peng-Qian-08,Liao-Li-11,Nono} for more results on octonionic regular functions.


 However, a significant limitation of monogenic (or regular) functions is that powers of the hypercomplex variable are not themselves monogenic (or regular), hindering the development of a more comprehensive theory. To address this shortcoming, Gentili and Struppa \cite{Gentili-Struppa-07} in 2006 introduced the concept of \textit{slice regularity} for functions of one quaternionic variable. Inspired by a concept of Cullen, this approach allows for a generalization of the notion of regularity and effectively solved the issue of powers. Slice regularity was subsequently extended to Clifford algebras \cite{Colombo-Sabadini-Struppa-09,Gentili-Struppa-08} and octonions \cite{Gentili-Struppa-10}, which led to a unification of the existing function theories under a broader framework for slice regular functions.

 Based on  the well-known Fueter construction, the theory of slice regular functions was later generalized by Ghiloni and Perotti \cite{Ghiloni-Perotti} to real alternative $\ast$-algebras, a class of algebras that includes the octonions. In this context, the theory of slice regular functions has been significantly developed  with foundational results extending to real alternative $\ast$-algebras, e.g. \cite{Ghiloni-Perotti-14,Ghiloni-Perotti-Stoppato-17,Ghiloni-Perotti-Stoppato-17adv,Ghiloni-Recupero-16,Ghiloni-Recupero-18,Ren-16}. These developments are particularly important as they allow for the unification of the theory of slice monogenic and slice regular functions, providing a more general framework for hypercomplex analysis.

In recent years, the concept of \textit{generalized partial-slice monogenic functions} has emerged \cite{Xu-Sabadini}, offering a unified framework that encompasses both monogenic and slice monogenic functions over Clifford algebras. See   \cite{Ding-Xu,Ding,Huo-24,Xu-Sabadini-3,Xu-Sabadini-2,Xu-Sabadini-4} for more results in this setting. This approach has   further broadened   the scope of hypercomplex function theory, particularly in the context of real alternative  $\ast$-algebras \cite{Ghiloni-Stoppato-24-1,Ghiloni-Stoppato-24-2} where   the notion of \textit{$T$-regular functions}  was proposed. Although these generalized theories have been studied in depth within the context of associative algebras, they have been less explored in non-associative algebras such as the octonions.
Hence,   the authors in \cite{Xu-Sabadini-25-o} continued to investigate generalized partial-slice monogenic functions from  the associative case to the non-associative alternative algebra of octonions, which  sheds  light on the present work.

In 1984, Ryan \cite{Ryan-84}  extended  monogenic functions in  Clifford analysis to complex, finite-dimensional, but associative algebras with identity. In 2020, Alpay, Paiva, and Struppa \cite{Alpay} developed the theory of hyperholomorphic functions with values   in an associative (real or  complex) Banach algebra.   In the study of slice analysis,   Perotti \cite{Perotti-22} in 2022 introduced the   Cauchy-Riemann operator $\overline{\partial}_{\mathcal{B}}$ over  real alternative algebras.
More recently, in 2024,  Ghiloni and Stoppato \cite{Ghiloni-Stoppato-24-2} reintroduced  monogenic functions  on domains  in hypercomplex subspaces and valued in  real alternative $\ast$-algebras, although their discussion of foundational results was limited to the associative case.
In summary, hyperholomorphic (or monogenic) functions with values   in general  algebras has received  considerable attention, but   most studies have been restricted to associative algebras.

This paper extends the study of monogenic functions to any finite-dimensional space, from associative algebras to general real alternative $\ast$-algebras, including the non-associative case such as octonions. In particular, we develop fundamental properties for monogenic functions on
domains in hypercomplex subspaces with values in real alternative $\ast$-algebras, such as the Cauchy (and Cauchy-Pompeiu) integral formula and the Taylor series expansion. Naturally, the main difficulty in this paper is dealing with the non-commutativity and especially the non-associativity of multiplications.

The paper is structured as follows. In Section 2, we review the basic properties of real alternative $\ast$-algebras and provide an overview of monogenic functions on hypercomplex subspaces. In particular, we establish an identity theorem (Theorem \ref{Identity-theorem}) for monogenic   functions on hypercomplex subspaces  of real alternative $\ast$-algebras. Section 3 is dedicated to the development of the Cauchy-Pompeiu integral formula (Theorem \ref{CauchyPompeiuslice}), along with some consequences, including the mean value theorem and maximum modulus principle. As an application, in the setting of real alternative $\ast$-algebras,   we define the  Teodorescu transform and show it can be viewed as a right inverse of the   Cauchy-Riemann operator $\overline{\partial}_{\mathcal{B}}$; see Theorem \ref{Cauchy-Pompeiu-inverse}.
Finally, in Section 4, we present the Taylor series expansion for monogenic functions in Theorem \ref{Taylor-lemma-T}  by introducing monogenic Fueter polynomials and applying a Cauchy-Kovalevskaya extension of real analytic functions over real alternative $\ast$-algebras.

 \section{Preliminaries}
In this section, we first collect some preliminary results for alternative algebras and   recall   the notion  of  monogenic   functions on hypercomplex subspaces, with a discussion on the identity theorem. For more details, we refer to  \cite{Okubo,Schafer}   on  alternative algebras and \cite{Ghiloni-Stoppato-24-2,Perotti-22} on monogenic   functions on a hypercomplex subspace of real alternative $\ast$-algebras.

\subsection{ Real alternative $\ast$-algebras and hypercomplex subspaces}
\noindent Let $\mathbb{A}$ be a real algebra with a unity. A real algebra $\mathbb{A}$ is said to be alternative if the \textit{associator}
$$[a,b,c]:= (ab)c-a(bc)$$ is a trilinear, alternating function in  the variables $a,b,c\in \mathbb{A}$.

All  alternative algebras   obey some weakened associative laws:
\begin{itemize}
\item
A theorem of E. Artin asserting that
\textit{the subalgebra generated by two elements of $\mathbb  A$ is associative}.
\item
The Moufang identities:
$$a(b(ac)) = (aba)c,\qquad ((ab)c)b = a(bcb), \qquad (ab)(ca) = a(bc)a,$$
for  $a,b,c\in\mathbb A$.
\end{itemize}
A real algebra $\mathbb  A$  is called  $\ast$-algebra  if $\mathbb  A$  is equipped with an anti-involution (or called $\ast$-involution), which is a real linear map $^{c}:\mathbb  A \rightarrow \mathbb  A$,  $a\mapsto a^{c}$ satisfying  that $a^{c}= a$ for $a\in \mathbb{R} \subseteq \mathbb  A $ and
$$ (a^{c})^{c} = a,  \ (ab)^{c} =b^{c}a^{c},  \quad a,b \in\mathbb A.$$
\begin{assumption}
Assume that $(\mathbb A, +,  \cdot, ^{c})$ is an alternative real  $\ast$-algebra with a unity,  of finite dimension $d>1$ as   a real vector space, and  equipped with an anti-involution $^{c}$.
Additionally, we endow  $\mathbb A$   with the natural topology and differential structure as a real vector space.
\end{assumption}

\begin{example} [Division algebras]
The classical division algebras of  the complex numbers $\mathbb{C}$, quaternions $\mathbb{H}$ and  octonions $\mathbb{O}$ are real  $\ast$-algebras, where $\ast$-involutions are  the standard conjugations of  complex numbers, quaternions, and  octonions, respectively.
\end{example}
\begin{example} [Clifford algebras]
 The   Clifford algebra   $\mathbb{R}_{0,m}$ is an associative $\ast$-algebra,  which is generated  by the standard orthonormal basis $\{e_1,e_2,\ldots, e_m\}$ of the $m$-dimensional real Euclidean space  $\mathbb{R}^m$ by assuming   $$e_i e_j + e_j e_i= -2\delta_{ij}, \quad 1\leq i,j\leq m,$$
where    $\ast$-involution is given by the Clifford conjugation.
\end{example}
\begin{example}[Dual quaternions]
 The dual quaternions $\mathbb{DH}$    is an associative  $\ast$-algebra defined as $\mathbb{H}+\epsilon \mathbb{H}$  with
 $g^{c}=g_1^{c}+\epsilon g_2^{c}$ and
\begin{itemize}
\item
Addition:  $g+h=g_1+h_1+\epsilon (g_2+ h_2),$
\item
Multiplication: $gh=g_1h_1+   \epsilon (g_1h_2+g_2h_1 ),$
\end{itemize}
for all $g=g_1+\epsilon g_2, h=h_1+\epsilon h_2\in \mathbb{DH},$ where $ g_1, g_2,h_1,h_2\in \mathbb{H}$, $\ast$-involution is the standard conjugates of  quaternions, $\epsilon$ commutes with every element of $\mathbb{DH}$   and  $\epsilon^{2}=0$.
\end{example}

For each element $x$ in the $\ast$-algebra ${\mathbb A}$, its \textit{trace}  is  given by $t(x):= x+x^{c}\in {\mathbb A}$ and its  (squared) \textit{norm} is $n(x):= xx^{c}\in {\mathbb A}.$
 Denote  the  sphere  of the imaginary units of $\mathbb A$ compatible with the  $\ast$-algebra structure  by
 $$\mathbb{S}_{{\mathbb A}}:= \{x  \in \mathbb A : t(x)=0,  n(x)=1  \}.$$
\begin{assumption}
Assume  $\mathbb{S}_{{\mathbb A}}\neq \emptyset$.
\end{assumption}
For each $J \in \mathbb{S}_{{\mathbb A}}$,   denote by $$\mathbb{C}_{J}:=\langle1,J\rangle \cong \mathbb{C},$$ the subalgebra of ${\mathbb A}$ generated by $1$ and $J$.
Denote the quadratic cone of ${\mathbb A}$ by
$$Q_{{\mathbb A}}:= \mathbb{R} \cup \big\{x \in {\mathbb A} \mid t(x) \in  \mathbb{R}, \ n(x)\in \mathbb{R}, \  4\,  n(x)>t(x)^{2} \big\}.$$
It is known that
$$  Q_{\mathbb A}=\bigcup_{J\in\mathbb S_{\mathbb A}} \mathbb C_J,$$
 and
 $$\mathbb C_I \cap \mathbb C_J=\mathbb R, \qquad  I, J\in\mathbb S_{\mathbb A}, \ I\neq \pm J.$$

\begin{definition}
A fitted basis of $\mathbb A$ is an ordered basis $(w_0, w_1,  \ldots , w_d)$ of $\mathbb A$ such that $w_s^{c} =\pm w_s$
for all $s\in \{0, \ldots, d\}$.
\end{definition}

\begin{lemma} (\cite[Lemma 2.8]{Ghiloni-Stoppato-24-2})
Fix $m\geq 1$. If $v_1,  \ldots, v_m \in \mathbb{S}_{{\mathbb A}}$ are linearly independent, then $(1,v_1,  \ldots, v_m )$ can
be completed to a fitted basis of  $\mathbb A$.
\end{lemma}

\begin{definition}
Let $M$ be a real vector subspace of our $\ast$-algebra $\mathbb A$. An ordered real vector
basis $(v_0,v_1,  \ldots, v_m)$  of $M $ is called a hypercomplex basis of $M$ if: $m \geq 1$; $v_0 =1$; $v_s \in \mathbb{S}_{{\mathbb A}}$
and $v_sv_t = -v_tv_s$ for all distinct $s, t\in \{1,  \ldots, m\}$. The subspace $M$ is called a hypercomplex
subspace of $\mathbb A$ if $\mathbb{R}  \subsetneq   M  \subseteq Q_{{\mathbb A}}$.
\end{definition}

Equivalently, a basis $(v_0=1,v_1,  \ldots, v_m)$ is a hypercomplex basis if, and only if, $t(v_s)=0, n(v_s)=1$
and $t(v_s v_t^{c})=0$ for all distinct $s, t\in \{1,\ldots , m\}$.

\begin{assumption}\label{ass} Within the alternative $\ast$-algebra $\mathbb A$, we fix a hypercomplex subspace $M$, with a fixed hypercomplex basis  $\mathcal{B}=(v_0,v_1,  \ldots, v_m)$. Moreover, $\mathcal{B}$ is completed to a real vector basis $\mathcal{B}' =(v_0, v_1,  \ldots , v_d)$ of  $\mathbb A$ and  $\mathbb A$ is endowed with the standard Euclidean scalar product $\langle\cdot,\cdot\rangle$ and norm $ | \cdot |$ associated to $\mathcal{B}'$. If this is the case, then we have for all $x, y\in M$
$$t(xy^{c})=t(y^{c}x)=2 \langle x,y\rangle,$$
$$n(x)=n(x^c)=|x|^{2}.$$
 \end{assumption}
 We remark that every hypercomplex subspace  admits a hypercomplex basis   \cite[\S \ 3]{Perotti-22} and refer the reader to \cite[Theorem 2.4]{Ghiloni-Stoppato-24-1} for the rationality of Assumption \ref{ass}.

\begin{example}\label{example:strong-regular}  The best-known examples of    hypercomplex subspaces   in real alternative $\ast$-algebras   are given by
\begin{eqnarray*}
M=
\begin{cases}
\mathbb H, \qquad  \quad \qquad   \, \mathbb A=\mathbb H, \qquad      Q_{\mathbb A}=\mathbb H,
\\
\mathbb O, \qquad  \quad \qquad  \, \mathbb A=\mathbb O, \qquad     Q_{\mathbb A}=\mathbb O,
\\  \mathbb{R}^{m+1}, \qquad  \quad \ \,  \mathbb A=\mathbb{R}_{0, m}, \ \ \  Q_{\mathbb A}\supseteq\mathbb R^{m+1}.
\end{cases}
\end{eqnarray*}
This allows to  unify the  theory of monogenic functions in quaternionic analysis, octonionic
analysis, and Clifford analysis. Furthermore, this theory includes the $\psi$-hyperholomorphic functions   when  the hypercomplex subspaces   in real alternative $\ast$-algebras $\mathbb{A}$, such as $\mathbb{H},  \mathbb{R}_{0,m}$, is chosen to be the so-called \textit{structural set} $\psi$; see  \cite{Shapiro} and references therein.

In addition, we can consider more  monogenic functions  theories, such as  $M=\mathbb H$   when $ \mathbb A=\mathbb{DH}$, where $Q_{\mathbb A}$ is a $6$-dimensional semialgebraic subset of $\mathbb{DH}$.
\end{example}

In this paper, we shall make use of the following useful properties  which are  immediate consequences of the Artin  theorem.
\begin{proposition}\label{artin-inverse}(\cite[p. 38]{Schafer})
For any $x, y \in \mathbb{A},$ if x is invertible, then it holds that
$$[x^{-1},x,y]=0.$$
\end{proposition}
\begin{proposition}\label{real}
For any $r\in \mathbb{R}$ and $x, y \in \mathbb{A},$ it holds that
$$[r,x,y]=0.$$
\end{proposition}
\begin{proposition}\label{artin}
For any $x\in M$ and $y \in \mathbb{A},$ it holds that
$$[x,x,y]=[x^{c},x,y]=0.$$
\end{proposition}
 At the end of this subsection, we recall a useful result from \cite[Remark 2.27]{Ghiloni-Stoppato-24-2}.
\begin{proposition}\label{norm}
  There exists some $C\geq1$ such that
$$|xy|\leq C | x |  | y |, \quad x\in M,y \in \mathbb{A}.$$
\end{proposition}

\subsection{Monogenic functions on hypercomplex subspaces}
 Recall that $\mathcal{B}=(v_0=1,v_1,  \ldots, v_m)$ is  a fixed hypercomplex basis of  the hypercomplex subspace $M$.  Throughout this paper,  an element $(x_0,x_1,\ldots,x_m)\in \R^{m+1}$ will be identified with an element in $M$,  via
$$(x_0,x_1,\ldots,x_m)  \mapsto  x= \sum_{s=0}^{m}x_sv_s =\sum_{s=0}^{m}v_s x_s. $$
 We consider the functions  $f:\Omega\longrightarrow  \mathbb{A}$, where $\Omega\subseteq \R^{m+1}$ is a domain (i.e., connected open set). As usual, denote $\mathbb{N}=\{0,1,2,\ldots\}$.
 For $k\in \mathbb{N}\cup \{\infty\}$,    denote by $C^{k}(\Omega,\mathbb{A})$ the set of all functions  $f (x)=\sum_{s=0}^{d}  v_{s}   f_{s}(x)$ with real-valued $f_{s}(x)\in  C^{k}(\Omega)$.

In this  subsection,   we  recall the definition of  monogenic functions.

\begin{definition}[Monogenic  function] \label{monogenic-Clifford}
Let $\Omega $ be a domain   in $M$ and let  $f\in C^{1}(\Omega,\mathbb{A})$.  The function $f (x)=\sum_{s=0}^{d}  v_{s}   f_{s}(x)$ is called  left monogenic  with respect to $\mathcal{B}$ in $\Omega $ if
$$ \overline{\partial}_{\mathcal{B}}f(x):=\sum _{s=0}^{m}v_{s} \frac{\partial f}{\partial x_{s}}(x)
= \sum _{s =0}^{m} \sum _{t=0}^{d}   v_{s}v_{t} \frac{\partial f_{t}}{\partial x_{s}}(x)=0, \quad x\in \Omega. $$
Similarly, the function $f$ is called  right  monogenic  with respect to $\mathcal{B}$  in $\Omega $ if
$$ f(x)\overline{\partial}_{\mathcal{B}}:=\sum _{s=0}^{m}  \frac{\partial f}{\partial x_{s}}(x)v_{s}= \sum _{s =0}^{m} \sum _{t=0}^{d}   v_{t}v_{s}
\frac{\partial f_{t}}{\partial x_{s}}(x)=0, \quad x\in \Omega. $$
\end{definition}
Denote by $\mathcal {M}(\Omega, \mathbb{A})$ or $\mathcal {M}^{L}(\Omega, \mathbb{A})$  (resp. $\mathcal {M}^{R}(\Omega, \mathbb{A})$)   the function class of  all  left  (resp. right) monogenic functions   $f:\Omega \rightarrow \mathbb{A}$.

 When $\mathbb{A}$ is  associative, such as   $\mathbb{A}=\mathbb{R}_{0,m}$,      $\mathcal {M}^{L}(\Omega, \mathbb{A})$ is   a right $\mathbb{R}_{0,m}$-module.
Meanwhile, when  $\mathbb{A}$ is     non-associative, such as $\mathbb{A}=\mathbb{O}$,  $\mathcal{M}^{L}(\Omega, \mathbb{A})$ is generally  not a right   $\mathbb{O}$-module.

The  partial differential operator with constant coefficients  in Definition  \ref{monogenic-Clifford}
$$\overline{\partial}_{\mathcal{B}}=\overline{\partial}_{x,\mathcal{B}}= \sum_{s=0}^{m}v_s\partial_{s},
\quad \partial_{s}=\frac{\partial}{\partial x_s},$$
is called the   Cauchy-Riemann operator  induced by $\mathcal{B}$.

The conjugated Cauchy-Riemann operator induced by  $\mathcal{B}$ is defined as
$$  \partial_{\mathcal{B}} f(x):=\partial_{0}f (x)- \sum _{s=1}^{m}v_{s}\partial_{s} f (x), \quad f\in C^{1}(\Omega,\mathbb{A}),$$
and
$$ f(x) \partial_{\mathcal{B}}:=\partial_{0}f (x)- \sum _{s=1}^{m}\partial_{s} f (x)v_{s}, \quad f\in C^{1}(\Omega,\mathbb{A}).$$
Define  the Laplacian operator on $M$  induced by  $\mathcal{B}$ as
$$   \Delta_{\mathcal{B}} f(x):=  \sum _{s=0}^{m} \partial_{s}^{2} f (x), \quad f\in C^{2}(\Omega,\mathbb{A}),$$
which is independent of  the choice of the  hypercomplex basis $\mathcal{B}$ of $M$.

\begin{example}\label{Cauchy-kernel-example}
Consider the Cauchy kernel
$$E(x):=\frac{1}{\sigma_{m}}\frac{x^{c} }{|x|^{m+1}}, \quad x \in\Omega=M\setminus \{0\},$$
where $\sigma_{m}=2\frac{\Gamma^{m+1}(\frac{1}{2}) }{\Gamma (\frac{m+1}{2})} $ is the surface area of the unit ball in $\mathbb{R}^{m+1}$. \\
Then $E \in \mathcal {M}^{L}(\Omega,  \mathbb{A}) \cap{\mathcal{M}}^{R}(\Omega,  \mathbb{A}).$
\end{example}

The proof of the left (and right) monogenicity of the Cauchy kernel is  omitted here
since it  could be included in   \cite[Lemma 3.13]{Ghiloni-Stoppato-24-2}, where the product of at most two non-real elements does not involve associativity.

\begin{remark}\label{C1-C00}{\rm
All functions  in $\mathcal {M}(\Omega,  \mathbb{A})$ are $ C^{\infty}(\Omega,  \mathbb{A})$ by Theorem \ref{Cauchy-slice} below, whose proof   is independent of what follows in the rest of this section.}  \end{remark}

For $\mathrm{k}=(k_0,k_1,\ldots,k_{m})\in\mathbb{N}^{m+1}$,  denote $|\mathrm{k}|:=k_0+k_1+\cdots+k_{m}$ and $\mathrm{k}!=k_0!k_1!\cdots k_m!$.  For $f \in C^{|\mathrm{k}|}(\Omega,\mathbb{A})$,  define
\begin{equation*}
\partial_{\mathrm{k}}f( x)= \frac{\partial^{|\mathrm{k}|} }{\partial_{x_{0}}^{k_0}\partial_{x_{1}}^{k_1}\cdots\partial_{x_{m}}^{k_m} }f( x).
\end{equation*}

By definition and Remark \ref{C1-C00}, we  have
\begin{proposition}\label{preserving-monogenic}
Let $\Omega\subseteq M$ be a  domain and $f\in \mathcal{M}(\Omega,  \mathbb{A})$.  Then, for any $y\in M$ and  $\mathrm{k}\in\mathbb{N}^{m+1}$, we have
$$f_{y}(\cdot):=f(y-\cdot) \in \mathcal {M}(\Omega_{y},  \mathbb{A}), \quad \partial_{\mathrm{k}}f \in \mathcal {M}(\Omega,  \mathbb{A}),$$
where $\Omega_{y}=\{x\in M:   y-x \in \Omega\}$.
\end{proposition}

\begin{proposition}\label{GSM-Harm}
A function $f\in \mathcal {M}(\Omega, \mathbb{A})$  is necessarily harmonic in $\Omega$, hence   real  analytic.
\end{proposition}
\begin{proof}
Let $f\in \mathcal {M}(\Omega, \mathbb{A})$.  By Remark \ref{C1-C00}, we have $f\in C^{2}(\Omega,  \mathbb{A})$,  then by \cite[Proposition 5]{Perotti-22}
$$   \Delta_{\mathcal{B}} f(x)= (\partial_{\mathcal{B}}  \overline{\partial}_{\mathcal{B}})f(x)=
\partial_{\mathcal{B}} (\overline{\partial}_{\mathcal{B}}f(x))=0,$$
which completes the proof.
\end{proof}
The identity theorem for real  analytic  functions   gives the following:
\begin{proposition}\label{Identity-theorem-monogenic-harmonic}
Let $\Omega\subseteq M$ be a  domain and $f\in \mathcal{M}(\Omega, \mathbb{A})$.
If $f$ equals to zero in a ball  of $M$, then $f\equiv 0$ in  $\Omega$.
\end{proposition}

Denote by $\mathcal{Z}_{f}(\Omega)$  the zero set of the function $f:\Omega\subseteq M\rightarrow \mathbb{A}$.
\begin{theorem} \label{Identity-lemma}
Let $\Omega\subseteq M $ be a  domain and $f: \Omega\rightarrow \mathbb{A}$ be a   monogenic function.
If   $\mathcal{Z}_{f}(\Omega)$ is a   $m$-dimensional smooth manifold $N$, then $f\equiv0$ in  $\Omega$.
\end{theorem}
\begin{proof}
We can follow  the strategy as in  Clifford case; see e.g. \cite[Theorem 9.27]{Gurlebeck}. For completeness, we give its details here.
Let $y$ be an arbitrary point in $N \subseteq \Omega$ and
$$x( \theta )  =x( \theta_1,\ldots,\theta_m)=\sum _{s=0}^{m} x_{s}(\theta) v_s$$
  be a parametrization of $N$ in a neighborhood of $y$ with
 $x(0)=y$. Due to that  $ f(x(\theta))=0$ for all $\theta$ and   that the functions $x_s(\theta)$ are real-valued, we have
  \begin{equation*}
  \sum _{s=0}^{m} \frac{\partial x_{s}}{\partial \theta_{t}} \frac{\partial f }{\partial x_{s}}(y)=0, \quad t=1,\ldots,m.
  \end{equation*}
Since $N$ is $m$-dimensional, we  assume that  rank$(\frac{\partial x_{s}}{\partial\theta_{t}})=m.$  Hence, without loss of generality, there exist some  $r_1,\ldots, r_m\in \mathbb{R}$ such that
 $$ \frac{\partial f}{\partial x_{s}}(y)=r_s \frac{\partial f}{\partial x_{0}}(y), \quad s=1,\ldots,m.   $$
Hence
 $$\overline{\partial}_{\mathcal{B}} f(y)= (1+ \sum _{s=1}^{m}  r_s v_s  )\frac{\partial f }{\partial x_{0}}(y)=0.$$
Since $ 1+ \sum _{s=1}^{m}  r_s v_s \in M\setminus \{0\} $ has an inverse, we have  $\frac{\partial f}{\partial x_{0}}(y)=0$, and then
 $$\frac{\partial f}{\partial x_{s}}( y)=0, \quad s=1,\ldots,m.$$
which shows  that all first-order derivatives of the function $f$ vanish in $N$.
Considering   $\frac{\partial f }{\partial x_{s}}$ $(s=0,1,\ldots,m)$ and  repeating the process above, we infer  that all derivatives of $f$ vanish in $N$. Consequently,   all  Taylor coefficients   for the real analytic function $f$ at some point $y$ vanish,  so that $f$   equals to zero in a  ball  in $\Omega$. Finally,  Proposition \ref{Identity-theorem-monogenic-harmonic} concludes that $f=0$ in  $\Omega$, as desired.
 \end{proof}

Theorem \ref{Identity-lemma} can be reformulated as follows.
\begin{theorem}  {\bf(Identity theorem)}\label{Identity-theorem}
Let $\Omega\subseteq M$ be a   domain and $f,g:\Omega\rightarrow \mathbb{A}$ be    monogenic functions.
If   $f=g$ on a   $m$-dimensional smooth manifold in $\Omega$, then $f\equiv g$ in  $\Omega$.
\end{theorem}

\subsection{Integrals  on a hypercomplex subspace}
In  this subsection,   we   fix some notations on  integrals over   real alternative $\ast$-algebras, which shall be used in next section.
\begin{definition}\label{integral}
Let    $\Omega $ be a   domain in $M$. For $f=\sum_{t=0}^{d}f_tv_t :  \Omega $  $\rightarrow \mathbb{A}  $ with real-valued integrable  components  $f_t$, define
$$\int_{\Omega} f dV:=\sum_{t=0}^{d}v_t \big(\int_{\Omega}f_t dV\big),  $$
where $dV$ stands  for  the   classical   Lebesgue  volume element  in $\mathbb{R}^{m+1}$.
\end{definition}
Note that, in the case of the associative  algebras, Definition \ref{integral}  has been given by Ghiloni and Stoppato in \cite[Definition 3.7]{Ghiloni-Stoppato-24-2}. By definition, the integral  in Definition \ref{integral}  has basic  properties below, whose proof      is the same as that  in     \cite[Proposition 3.8]{Ghiloni-Stoppato-24-2} and is omitted  here.
\begin{proposition}\label{absolute-domain}
Let $f, g$ be as in Definition \ref{integral}. The following properties hold true for all $a, b \in \mathbb{A}$:\\
(i)  $\int_{\Omega} f dV=\int_{\Omega_1} f dV+\int_{\Omega_2} f dV$, where $ \Omega=\Omega_1 \sqcup \Omega_2$;\\
(ii) $\int_{\Omega} (fa+gb) dV=(\int_{\Omega} f dV) a+(\int_{\Omega} gdV) b;$\\
(iii) $\int_{\Omega} (af+bg) dV=a(\int_{\Omega} f dV)  +b(\int_{\Omega} gdV);$\\
(iv)  $|\int_{\Omega} f dV| \leq\int_{\Omega} |f| dV$.
\end{proposition}

Similarly, we have
 \begin{definition}\label{integral-surface}
Let $\Gamma$ be a  surface in $M$. For $f=\sum_{t=0}^{d}f_tv_t : \Gamma\rightarrow \mathbb{A}  $ with real-valued integrable  components  $f_t$, define
$$\int_{\Gamma} f dS:=\sum_{t=0}^{d}v_t \big(\int_{\Gamma}f_t dS\big),  $$
where $dS$ stands  for  the   classical   Lebesgue  surface element  in $\mathbb{R}^{m+1}$.
\end{definition}
\begin{proposition}\label{absolute-boundary}
Let $f, g$ be as in Definition \ref{integral-surface}. The following properties hold true for all $a, b \in \mathbb{A}$:\\
(i)  $\int_{\Gamma} f dS=\int_{\Gamma_1} f dS+\int_{\Gamma_2} f dS$, where $\Gamma=\Gamma_1 \sqcup \Gamma_2$;\\
(ii) $\int_{\Gamma} (fa+gb) dS=(\int_{\Gamma} f dS) a+(\int_{\Gamma} gdS) b;$\\
(iii) $\int_{\Gamma} (af+bg) dS=a(\int_{\Gamma} f dS)  +b(\int_{\Gamma} gdS);$\\
(iv)  $  |\int_{\Gamma} f dS | \leq \int_{\Gamma} |f| dS.$
\end{proposition}

\section{Cauchy-Pompeiu integral formula}
\label{Sect4}
To formulate the Cauchy-Pompeiu integral formula over the alternative  algebras, we need to establish   some technical lemmas.
In the  octonionic setting, we mention that the corresponding  results  in following Lemmas \ref{E-lemma},  \ref{Cauchy-formula-lemma-2-0}, \ref{Cauchy-formula-lemma}, and \ref{Cauchy-formula-lemma-E} were obtained in \cite{Xu-Sabadini-25-o},   but in a very general   partial-slice form.

\begin{lemma}\label{E-lemma}
Let    $\Omega $ be a   domain in $M$. If $\phi(x)=\sum_{s=0}^m \phi_s(x) v_s \in C^1(\Omega, M)$ satisfies
\begin{equation}\label{i-j-E-1}
\partial_{ s} \phi_t= \partial_{ t}\phi_s, \quad 1\leq s,t\leq m,
\end{equation}
then for all $a\in \mathbb{A}$
$$ \overline{\partial}_{\mathcal{B}} (\phi  a)=(\overline{\partial}_{\mathcal{B}} \phi )a.$$
\begin{proof}
Recall that for all $a, b\in \mathbb{A}$,
$$[v_0, b, a]=0, $$
thus we have
\begin{eqnarray*}
[\overline{\partial}_{\mathcal{B}}, \phi, a]
&=&\sum_{s=1}^{m}[v_s, \partial_{ s} \phi, a]
 \\
 &=&  \sum_{s,t=1}^{m}[v_s, v_t, a] \partial_{ s} \phi_{t}
 \\
 &=& \sum_{1\leq s < t \leq m} [v_s, v_t, a] ( \partial_{ s}\phi_{t}-\partial_{ t} \phi_{s})
 \\
  &=& 0,
\end{eqnarray*}
where  the last equality  follows from   (\ref{i-j-E-1}).
Hence,
 $$ (\overline{\partial}_{\mathcal{B}}  \phi)  a- \overline{\partial}_{\mathcal{B}} (\phi )a=0,$$
 i.e.
$$ \overline{\partial}_{\mathcal{B}} (\phi  a)=(\overline{\partial}_{\mathcal{B}} \phi )a.$$
The proof is complete.
\end{proof}
\end{lemma}

\begin{lemma}\label{Cauchy-formula-lemma-2-0}
  For any $a\in \mathbb{A}$, it holds that
$$ \overline{\partial}_{\mathcal{B}} (E(x)a)=0, \quad  x \neq 0.$$
\end{lemma}
\begin{proof}
 Recall the Cauchy kernel in Example \ref{Cauchy-kernel-example}
 $$E(x)=E_{0}(x)-\sum_{s=1}^{m}E_{s}(x)v_s, \quad E_{s}(x)=\frac{1}{\sigma_{m}}\frac{x_s}{| x|^{m+1}}, \ s=0,1,\ldots,m.$$
It is immediate that
 \begin{equation}\label{Eij}
\partial_{ t} E_{s}= \frac{1}{\sigma_{m}}(\frac{\delta_{st}}{|x|^{m+1}}-(m+1)\frac{x_tx_s}{| x|^{m+3}})=\partial_{ s} E_{t},\ s,t=1,\ldots,m,\end{equation}
which   satisfy (\ref{i-j-E-1}).
Hence   Lemma \ref{E-lemma} completes the proof. \end{proof}

\begin{lemma}\label{Cauchy-formula-lemma}
Let   $\Omega $ be a   domain in $M$. If $\phi =\sum_{s=0}^m \phi_s  v_s \in C^1(\Omega, M) $  satisfies
\begin{equation}\label{i-j}
\partial_{s}\phi_t= \partial_{t}\phi_s, \quad 1\leq s,t\leq m,
\end{equation}
 then  for any $a\in \mathbb{A}$
$$\sum_{s=0}^{m}[v_s, \overline{\partial}_{\mathcal{B}} \phi_s, a]=0.$$
\end{lemma}
\begin{proof}
Recall that $[v_0, b, a]=0, $ for all $a, b\in \mathbb{A}$ so that
\begin{eqnarray*}
\sum_{s=0}^{m}[v_s, \overline{\partial}_{\mathcal{B}} \phi_s, a]
&=&\sum_{s=1}^{m}[v_s, \overline{\partial}_{\mathcal{B}} \phi_s, a]
 \\
 &=&  \sum_{s =1}^{m} \sum_{ t=0}^{m-1} [v_s, v_t, a] \partial_{ t} \phi_{s} +\sum_{s=1}^{m}[v_s, v_m, a]\partial_{m} \phi_{s}
 \\
 &=&  \sum_{s =1}^{m} \sum_{ t=1}^{m-1} [v_s, v_t, a] \partial_{ t} \phi_{s} +\sum_{s=1}^{m-1}[v_s, v_m, a]\partial_{m} \phi_{s}.
 \end{eqnarray*}
In view of (\ref{i-j}), it holds that
$$\sum_{1\leq s,t\leq m-1}   [v_s,v_t, a]\partial_{t} \phi_s =0.$$
Hence, we have
\begin{eqnarray*}
\sum_{s=0}^{m}[v_s, \overline{\partial}_{\mathcal{B}} \phi_s, a]
&=&  \sum_{ t=1}^{m-1} [v_m, v_t, a] \partial_{ t} \phi_{m} +\sum_{s=1}^{m-1}[v_s, v_m, a]\partial_{m} \phi_{s}
\\
&=&    \sum_{ s=1}^{m-1} [v_m, v_s, a] \partial_{s} \phi_{m} +\sum_{s=1}^{m-1}[v_s, v_m, a]\partial_{m} \phi_{s}
\\
 &=&\sum_{ s=1}^{m-1} [v_m, v_s, a]( \partial_{s} \phi_{m}-\partial_{m} \phi_{s})
 \\
 &=&0,\end{eqnarray*}
which completes from  (\ref{i-j}) the proof.
\end{proof}

\begin{lemma}\label{Cauchy-formula-lemma-E}
  For any $a\in \mathbb{A}$, it holds that
$$\sum_{s=0}^{m}[v_s, \overline{\partial}_{\mathcal{B}} E_s(x), a]=0, \quad  x\neq 0.$$
\end{lemma}
\begin{proof}
Recalling  that the Cauchy kernel $E$   satisfies (\ref{Eij}),  we
 can conclude the proof by  Lemma \ref{Cauchy-formula-lemma}.
\end{proof}

Now we can describe  a Gauss theorem  over   real alternative $\ast$-algebras, which generalizes    \cite[Theorem 3.11]{Ghiloni-Stoppato-24-2} from associative algebras to alternative  algebras.
\begin{lemma}[Gauss]\label{Cauchy-formula-lemma-2}
Let     $\Omega$ be a bounded  domain in $M$ with smooth boundary $\partial \Omega$. If  $\phi=\sum_{s=0}^{m}\phi_sv_s \in C^1( \overline{\Omega},  M)$ and $f\in C^1( \overline{\Omega},\mathbb{A})$, then
 $$\int_{\partial \Omega}   \phi    (n f)dS= \int_{\Omega} \Big((\phi \overline{\partial}_{\mathcal{B}}  )f + \phi (\overline{\partial}_{\mathcal{B}} f)
  - \sum_{s=0}^{m} [v_s, \overline{\partial}_{\mathcal{B}} \phi_s, f]\Big)dV,
$$
where  $n$ is the unit exterior normal to $\partial \Omega$, $dS$ and  $dV$ stand  for  the   classical   Lebesgue surface  element and volume element  in $\mathbb{R}^{m+1}$, respectively.
\end{lemma}
\begin{proof}
Let $\phi=\sum_{s=0}^{m}\phi_sv_s \in C^1( \overline{\Omega},  M)$ and $f=\sum_{t=0}^{d}f_tv_t \in C^1(  \overline{\Omega},  \mathbb{A})$.
By the divergence theorem, we have for all real-valued functions $\phi_{s}, f_t\in C^1( \overline{\Omega})$
$$\int_{\partial \Omega}   \phi_{s}f_t   n_{k} dS
=\int_{\Omega}(   (\partial_{_k}\phi_{s} )f_t + \phi_{s} (\partial_{k}f_t) ) dV,\quad k=0,1,\ldots,m. $$
By multiplying by $v_k, k=0,1,\ldots,m,$  on both sides of the above formula, respectively, and  taking summation over $k$, we obtain
$$\int_{\partial \Omega}   \phi_{s}f_t  n dS =
\int_{ \Omega} ( (\overline{\partial}_{\mathcal{B}}  \phi_{s} )f_t+ \phi_{s} (\overline{\partial}_{\mathcal{B}}  f_t) )dV.$$
Then, by multiplying by $v_t, t=0,1,\ldots,d,$  from the right side,  we get
$$\int_{\partial \Omega}   \phi_{s}    n fdS =
\int_{\Omega}  ( (\overline{\partial}_{\mathcal{B}}  \phi_{s} )f  + \phi_{s} (\overline{\partial}_{\mathcal{B}}  f ) )dV,$$
which implies, by multiplying by $v_s, s=0,1,\ldots,m,$ from the left side and then taking  summation over $s$,
$$\int_{ \partial \Omega}   \phi    ( n f)dS
=\int_{ \Omega} \Big( \sum_{s=0}^{m} v_s ((\overline{\partial}_{\mathcal{B}} \phi_{s} )f) + \phi (\overline{\partial}_{\mathcal{B}} f)\Big) dV,$$
which gives
$$\int_{\partial \Omega}   \phi    (n f)dS
=\int_{ \Omega}\Big( (\phi \overline{\partial}_{\mathcal{B}}  )f  + \phi (\overline{\partial}_{\mathcal{B}} f) -
\sum_{s=0}^{m} [v_s,  \overline{\partial}_{\mathcal{B}}  \phi_s, f]\Big)dV,$$
 as desired.
 \end{proof}

 Now we can  prove  the Cauchy-Pompeiu integral  formula, which  subsumes \cite[Theorem 7.8]{Gurlebeck}    and
  \cite[Theorem 3.14]{Ghiloni-Stoppato-24-2}.
\begin{theorem}[Cauchy-Pompeiu]\label{CauchyPompeiuslice}
Let    $\Omega$ be a bounded domain in $M$ with smooth boundary $\partial \Omega$. If $f\in C^1( \overline{\Omega}, \mathbb{A}) $, then
 $$ f(x)=\int_{\partial \Omega}  E_{y}(x) (n(y)f(y)) dS(y)-
\int_{\Omega} E_{y}(x)  (\overline{\partial}_{\mathcal{B}} f(y)) dV(y), \quad  x \in \Omega,   $$
where  $n(y)$ is the unit exterior normal to $\partial \Omega$ at $y$,
 $dS$ and  $dV$ stand  for  the   classical   Lebesgue surface  element and volume element  in $\mathbb{R}^{m+1}$, respectively.
\end{theorem}
\begin{proof}
Given $x \in \Omega$,  denote $B(x,\epsilon):=\{y \in M: | y- x|<\epsilon\}$. Let $\phi( y)=E_{ x}( y)=E( x- y)=-E_{ y}( x)$. Then it holds that
$$\phi( y) \overline{\partial}_{y, \mathcal{B}} =0, \quad   y\in M\backslash \{x\}.$$
For $\epsilon$ small enough, we have by  Lemmas \ref{Cauchy-formula-lemma-E} and  \ref{Cauchy-formula-lemma-2} for $f\in C^1( \overline{\Omega}, \mathbb{A}) $
 \begin{eqnarray*}
&& \int_{\partial \Omega } \phi    ( n f)dS - \int_{\partial B( x,\epsilon)}   \phi    ( n f)dS\\
&=& \int_{ \Omega \setminus B( x,\epsilon)}\Big( (\phi \overline{\partial}_{\mathcal{B}}  )f + \phi (\overline{\partial}_{\mathcal{B}}f) -
\sum_{s=0}^{m} [v_s,  \overline{\partial}_{\mathcal{B}} \phi_s, f]\Big)dV
\\
&=&\int_{ \Omega \setminus B(x,\epsilon)} \phi (\overline{\partial}_{\mathcal{B}}f)   dV.
 \end{eqnarray*}
Recalling Proposition \ref{artin}, it follows that
 \begin{eqnarray*}
 \int_{\partial B( x,\epsilon) }   \phi    ( n f)dS
&=&   \frac{1}{\sigma_{m}}\int_{\partial B( x,\epsilon) }  \frac{  (y- x)^{c}}{| y- x|^{m+1}}   (\frac{ y- x}{| y- x|} f( y))dS( y)
\\
&=& \frac{1}{\sigma_{m}}\int_{\partial B( x,\epsilon) } \Big( \frac{(y- x)^{c}}{| y- x|^{m+1}}   \frac{ y- x}{| y- x|}\Big) f( y)dS( y)
\\
&=& \frac{1}{\epsilon^{m}\sigma_{m}}\int_{\partial B( x,\epsilon) }   f( y)dS( y)
\\
&\rightarrow &   f( x), \ \epsilon\rightarrow0.
 \end{eqnarray*}
Combining these two facts above, we infer
 \begin{eqnarray*}
 \int_{\partial \Omega}   \phi    ( n f)dS-f( x)
=\lim_{\epsilon\rightarrow0} \int_{\Omega  \setminus B( x,\epsilon)} \phi (\overline{\partial}_{\mathcal{B}}f)   dV
=\int_{\Omega  } \phi (\overline{\partial}_{\mathcal{B}}f)   dV,
\end{eqnarray*}
i.e.
 $$f( x)=\int_{\partial \Omega}  E_{ y}( x) ( n( y)f( y)) dS( y)-
\int_{\Omega} E_{ y}( x)  (\overline{\partial}_{\mathcal{B}}f( y)) dV( y).$$
 The proof is complete.
\end{proof}

\begin{definition}\label{Teodorescu}
Let    $\Omega$ be a bounded domain in $M$ with  smooth boundary $\partial \Omega$ and    $f\in C^1( \overline{\Omega}, \mathbb{A})$. Define the Teodorescu transform  as
 $$T[f](x)=-\int_{  \Omega}  E_{y}(x) f(y) dV(y), \quad  x \in \Omega,   $$
where  the integral   is   singular and exists in the   sense of  Cauchy's principal value,
i.e.
$$T[f](x)=- \lim_{\epsilon\rightarrow 0}\int_{  \Omega  \setminus B( x,\epsilon) }  E_{y}(x) f(y) dV(y).$$
Define  the Cauchy transform  as
 $$ C[f](x)=\int_{\partial \Omega}  E_{y}(x) (n(y)f(y)) dS(y),\quad  x\notin \partial \Omega.$$
\end{definition}

The $\Pi$-operator  plays a   major  role in complex and hypercomplex  analysis,  especially  in the theory of
generalized analytic functions in the sense of Vekua.  In  real alternative $\ast$-algebras,
we can also define $\Pi$-operator by
$$\Pi[f] (x)= \partial_{\mathcal{B}}  T[ f] (x),$$
 which   subsumes the Clifford analysis case \cite[Definition 1]{Gurlebeck-96} and the octonionic case \cite[Definition 4.1]{Krau-24}.
 The relevant properties of the $\Pi$-operator shall be discussed in details in a forthcoming  paper.

The Cauchy-Pompeiu integral  formula in Theorem \ref{CauchyPompeiuslice}   reads as
$$ C[f](x)+T[\overline{\partial}_{\mathcal{B}} f] (x) =f(x),\quad x\in \Omega.$$

In fact, the  Teodorescu transform can be viewed as  a right inverse of the   Cauchy-Riemann operator $\overline{\partial}_{\mathcal{B}}$,
 which  subsumes \cite[Theorem 8.2]{Gurlebeck}, \cite[Theorem 4.2]{Wang} and  \cite[Theorem 3.7]{Krau-24}.
\begin{theorem}\label{Cauchy-Pompeiu-inverse}
Let    $\Omega$ be a bounded domain in $M$ with  smooth boundary $\partial \Omega$ and    $f\in C^1( \overline{\Omega}, \mathbb{A})$. Then $T[ f] (x)$ is also differentiable in $\Omega$ with
 \begin{equation}\label{T-s}
\partial_s T[ f] (x)=-\int_{  \Omega} \partial_s E_{y}(x) f(y) dV(y)+v_s^{c} \frac{f(x)}{m+1}, \ s=0,1,\ldots,m.\end{equation}
In particular, we have
$$\overline{\partial}_{\mathcal{B}} T[ f] (x) =f(x),\quad x\in \Omega.$$
\end{theorem}
\begin{proof}
From the proof of \cite[Theorem 8.2]{Gurlebeck}, it holds that for any real-valued function $u\in C^1(  \overline{\Omega})$
$$\partial_s T[u] (x)=-\int_{  \Omega} \partial_s E_{y}(x) u(y) dV(y)+ \frac{v_s^{c}}{m+1}u(x), \ s=0,1,\ldots,m.$$
Then the general case (\ref{T-s}) of $\mathbb{A}$-valued functions can be done by addition due to that this is no issue of the associativity.

Now we use (\ref{T-s}) to show the Teodorescu transform is a right inverse of $\overline{\partial}_{\mathcal{B}}$.
To  this  end, we first note that by Proposition \ref{artin}
$$v_s(v_s^{c} f(x))=(v_sv_s^{c}) f(x)=f(x), \ s=0,1,\ldots,m,$$
and   that $E_{y}(x)$ satisfies (\ref{i-j-E-1}), and hence Lemma \ref{E-lemma} gives that
$$\sum_{s=0}^{m} [v_s,  \partial_s E_{y}(x), f(y)]=0.$$
Hence,
 \begin{eqnarray*}
 \overline{\partial}_{\mathcal{B}} T[ f] (x)
&=&-\int_{  \Omega} \sum_{s=0}^{m} v_s (\partial_s E_{y}(x) f(y)) dV(y)+\frac{1}{m+1}\sum_{s=0}^{m} v_s(v_s^{c} f(x))
\\
   &=&-\int_{  \Omega} \sum_{s=0}^{m} (v_s \partial_s E_{y}(x)) f(y)dV(y)+\frac{1}{m+1}\sum_{s=0}^{m}   f(x)
   \\
   &=&-\int_{  \Omega} \sum_{s=0}^{m} (v_s \partial_s E_{y}(x)) f(y)dV(y)+ f(x)
   \\
   &=&f(x),
\end{eqnarray*}
 which completes the proof.

\end{proof}

By Theorem  \ref{CauchyPompeiuslice}, we obtain
\begin{theorem}[Cauchy]\label{Cauchy-slice}
Let    $\Omega $ be a bounded domain in $M$ with smooth boundary $\partial \Omega $.  If $ f\in \mathcal{M}( \overline{\Omega}, \mathbb{A})$, then
 $$f(x)=\int_{\partial \Omega}  E_{y}(x) (n(y)f(y)) dS(y), \quad  x \in \Omega,   $$
where  $n(y)$ is the unit exterior normal to $\partial \Omega$ at $y$,
 $dS$ and  $dV$ stand  for  the   classical   Lebesgue surface  element and volume element  in $\mathbb{R}^{m+1}$, respectively.
\end{theorem}

For   $\mathrm{k} \in \mathbb{N}^{m+1}$, define
$$\mathcal{Q}_{ \mathrm{k}}  (x):=(-1)^{|\mathrm{k}|}  \partial_{\mathrm{k}} E(x).$$
\begin{corollary} \label{Cauchy-slice-high}
Let    $\Omega \subseteq M$ be a bounded domain  containing $0$ with smooth boundary $\partial \Omega $.  If $ f\in \mathcal{M}(  \overline{\Omega}, \mathbb{A})$, then
for all  $\mathrm{k} \in \mathbb{N}^{m+1}$
  $$ \partial_{\mathrm{k}} f(0)=\int_{\partial \Omega} \mathcal{Q}_{ \mathrm{k}}(y) (n(y)f(y)) dS(y),   $$
where  $n(y)$ is the unit exterior normal to $\partial \Omega$ at $y$,
 $dS$ and  $dV$ stand  for  the   classical   Lebesgue surface  element and volume element  in $\mathbb{R}^{m+1}$, respectively.
\end{corollary}

In view of Lemma \ref{Cauchy-formula-lemma-2-0},   an inverse of the Cauchy  integral formula holds.
\begin{theorem}\label{Cauchy-slice-inverse}
Let     $\Omega$ be a bounded domain in $M$ with smooth boundary $\partial \Omega$ and
 $f\in C( \partial \Omega, \mathbb{A}) $. Then
  $C[f] \in \mathcal{M}(M \setminus \partial \Omega,  \mathbb{A})$.
\end{theorem}

 As  in the classical case of monogenic functions,  Theorem \ref{Cauchy-slice} allows  to obtain several consequences, such as mean value theorem and maximum modulus principle. The proof can be given as in \cite{Xu-Sabadini-25-o}, we omit it here.
\begin{theorem}[Mean value theorem]\label{mean}
Let      $\Omega$ be a bounded  domain in $M$.  If $f\in \mathcal{M}(\Omega, \mathbb{A})$, then
 $$f(x)=\frac{1}{\sigma_{m}  \epsilon^{m}} \int_{\partial B(x,\epsilon) }  f(y) dS(y), \quad x \in B( x,\epsilon) ,  $$
where $B( x,\epsilon)  \subset \Omega$.
\end{theorem}

\begin{theorem}{\bf (Maximum modulus principle)}
Let      $\Omega$ be a bounded domain in $M$ and  $f\in \mathcal{M}(\Omega, \mathbb{A})$.   If $|f|$ has a relative maximum at some point in $\Omega$, then $f$ is constant.
\end{theorem}

\section{Fueter polynomials and Taylor series expansion}
In this section, we  introduce some monogenic Fueter polynomials which are   the building blocks of the Taylor expansion for monogenic functions
on hypercomplex subspaces of real alternative $\ast$-algebras.

First, we introduce  the Fueter variables   as in the classical Clifford case.
\begin{definition}\label{Fueter-variables}
 The so-called  Fueter variables are defined as
 $$ z_{\ell}( x)=   x_{\ell}-x_0 v_{\ell},  \quad \ell=1,2,\ldots,m.$$
  \end{definition}
By using Proposition \ref{artin}, direct calculation shows that
$$z_{\ell} ( x)\in \mathcal{ M}^{L}(M,M)\cap \mathcal{ M}^{R}(M,M),  \quad \ell=1,2,\ldots,m. $$

A natural idea is to use  Fueter variables in Definition \ref{Fueter-variables} to construct   monogenic Fueter polynomials as in the associative  case.  However, in view of the lack of associativity for general  alternative algebras $\mathbb{A}$, we need one useful result  concerning the multiplication of  ordered $n$ elements.

Given an  alignment $(a_1, a_2,\ldots, a_n)\in \mathbb{A}^{n}$, it is known that the multiplication of  ordered $n$ $(\geq2)$ elements
$a_1 a_2\cdots a_n$  has $ \frac{(2n-2)!}{n!(n-1)!}$  different associative orders. Denote by
$(a_{1}  a_{2} \cdots  a_{n})_{\otimes_{n}}$   the product of the ordered $n$ elements $(a_1, a_2,\ldots ,a_n)$ in a fixed associative order $\otimes_{n}$. In particular, denote the multiplication  from left to right by
$$(a_1 a_2 \cdots a_n)_{L}:=(\cdots((a_1 a_2)a_3) \cdots )a_n,$$
and the multiplication  from  right to left by
$$(a_1 a_2 \cdots a_n)_{R}:=a_1(  \cdots (a_{n-2}(a_{n-1} a_n))\cdots ).$$

\begin{proposition}\label{no-order}
Let $a, a_1,a_2,  \ldots ,a_m\in \mathbb{A} $ and $(j_1,j_2,\ldots, j_k) \in \{1,2,\ldots,m\}^{k}$,  repetitions being allowed. Then the following sum is independent of the chosen associative order $\otimes_{(k+1)}$
\begin{eqnarray}\label{sum}
  \sum_{(i_1,i_2, \ldots, i_k)\in \sigma} (a_{i_1} a_{i_2} \cdots a_{i_k}a)_{\otimes_{(k+1)}},
   \end{eqnarray}
where the sum runs over  all distinguishable permutations $\sigma$ of $(j_1,j_2,\ldots, j_k)$. \\
Hence, in particular, we have
$$\sum_{(i_1,i_2, \ldots, i_k)\in \sigma} (a_{i_1}  a_{i_2} \cdots  a_{i_k}a)_L
=\sum_{(i_1,i_2, \ldots, i_k)\in \sigma} (a_{i_1}  a_{i_2} \cdots  a_{i_k}a)_{R}.$$
\end{proposition}
\begin{proof}
Let $a\in \mathbb{A}$, $ x=\sum_{i=1}^{k}t_{i}a_{j_i}\in \mathbb{A}$, where $ t_{i}\in \mathbb{R}, j_i \in \{1,2,\ldots,m\}, i=1,2,\ldots,k$.
Observe  that, for a fixed associative order $\otimes_{(k+1)}$,
   the coefficient of $k_1!k_2!\cdots k_m!t_{1}t_{2}\cdots t_{k}$ in $(\underbrace{x x \cdots x}_{k} a)_{\otimes_{(k+1)}}$
  is given by the sum in (\ref{sum}), where $k_{\ell}$  is the appearing times of $\ell$ in $(j_1,j_2,\ldots, j_k)$, $\ell=1,\ldots,m$.

 For any associative order $\otimes_{(k+1)}$, it holds that  by Proposition  \ref{artin}
 $$(\underbrace{x x \cdots x}_{k} a)_{\otimes_{(k+1)}}=x^{k}a,$$
which means  the sum in (\ref{sum}) dees not depend on the associative order $\otimes_{(k+1)}$, which completes the proof.
\end{proof}
In the octonionic setting,  Proposition \ref{no-order}  was obtained  in the  case of    $a=1$ for Fueter variables by induction  \cite{Liao-Li-11}.

 For  $\mathrm{k}=(k_1,\ldots,k_{m})\in\mathbb{N}^{m}$, denote $\mathrm{k}!:=  k_1!\cdots k_m!$ and $|\mathrm{k}|:=k_1+k_2+\cdots+k_{m}$.
Proposition \ref{no-order} allows to  construct   monogenic Fueter polynomials in the general alternative setting.

\begin{definition}\label{definition-Fueter}
For  $\mathrm{k} \in \mathbb{N}^{m}$, let   $\overrightarrow{\mathrm{k}}:=(j_1,j_2,\ldots, j_k)$   be an alignment where the number of $1$ in the alignment is $k_1$, the number of $2$ is $k_2$, and the number of $m$ is $k_{m}$, where $k=|\mathrm{k}|, 1\leq j_1\leq j_2\leq \ldots\leq j_k\leq m$. Define
$$\mathcal{P}_{\mathrm{k}}( x) = \frac{1}{k!}  \sum_{(i_1,i_2, \ldots, i_k)\in \sigma(\overrightarrow{\mathrm{k}})} z_{i_1}  z_{i_2} \cdots  z_{i_k},$$
where the sum runs over the $\dfrac{k!}{\mathrm{k}!}$ distinguishable   permutations $\sigma(\overrightarrow{\mathrm{k}})$ of $\overrightarrow{\mathrm{k}}$. When $\mathrm{k}=(0,\ldots,0)=\mathbf{0}$, we set $\mathcal{P}_{ \mathbf{0}}(x)=1$.
\end{definition}

\begin{proposition} \label{Fueter-Mono}
  For all $\mathrm{k} \in \mathbb{N}^{m}$, it holds that $\mathcal{P}_{\mathrm{k}} \in \mathcal{ M}^{L}(M, \mathbb{A})\cap \mathcal{ M}^{R}(M, \mathbb{A}).$
\end{proposition}

\begin{proof}
First, let us show   $\mathcal{P}_{\mathrm{k}} \in \mathcal{ M}^{L}(M, \mathbb{A})$. By Definition \ref{definition-Fueter} and Proposition \ref{no-order}, we have
\begin{eqnarray*}
 k!x_0  \overline{\partial}_{\mathcal{B}}  \mathcal{P}_{\mathrm{k}}( x)
  &=& x_0 \sum_{(i_1,i_2, \ldots, i_k)\in \sigma(\overrightarrow{\mathrm{k}}) } \Big(\partial_0 (z_{i_1}  z_{i_2} \cdots  z_{i_k})_{R} +
  \sum_{s=1}^{m}v_s\partial_{s}(z_{i_1}  z_{i_2} \cdots  z_{i_k})_{R} \Big)
    \\
   &=&x_0 \sum_{(i_1,i_2, \ldots, i_k)\in \sigma(\overrightarrow{\mathrm{k}}) } \Big(  -\sum_{\ell=1}^{k}(z_{i_1}  z_{i_2} \cdots  z_{i_{\ell-1}}  v_{i_{\ell}}   z_{i_{\ell+1}} \cdots z_{i_k})_{R} \\
  & &\quad \quad \quad\quad\quad\quad \quad\quad +\sum_{\ell=1}^{k} v_{i_\ell}(z_{i_1}  z_{i_2} \cdots  z_{i_{\ell-1}} 1  z_{i_{\ell+1}} \cdots z_{i_k})_{R} \Big)
  \\
   &= &  \sum_{\ell=1}^{k} \Big(   \sum_{(i_1,i_2, \ldots, i_k)\in \sigma(\overrightarrow{\mathrm{k}}) }   (z_{i_1}  z_{i_2} \cdots  z_{i_{\ell-1}}  (-x_0 v_{i_{\ell}}  ) z_{i_{\ell+1}} \cdots z_{i_k})_{R} \\
  &&\quad   \quad\quad  -  \sum_{(i_1,i_2, \ldots, i_k)\in \sigma(\overrightarrow{\mathrm{k}}) } (- x_0v_{i_\ell}) (z_{i_1}  z_{i_2} \cdots  z_{i_{\ell-1}}  z_{i_{\ell+1}} \cdots z_{i_k})_{R}\Big)
 \\
   &= & \sum_{\ell=1}^{k}  \Big(   \sum_{(i_1,i_2, \ldots, i_k)\in \sigma(\overrightarrow{\mathrm{k}}) }   (z_{i_1}  z_{i_2} \cdots  z_{i_{\ell-1}}  z_{i_{\ell}}   z_{i_{\ell+1}} \cdots z_{i_k})_{R} \\
  &&\quad\quad\quad-   \sum_{(i_1,i_2, \ldots, i_k)\in \sigma(\overrightarrow{\mathrm{k}}) } z_{i_\ell}  (z_{i_1}  z_{i_2} \cdots  z_{i_{\ell-1}}  z_{i_{\ell+1}} \cdots z_{i_k})_{R} \Big)
 \\
   &=&   \sum_{\ell=1}^{k} \Big(   \sum_{(i_1,i_2, \ldots, i_k)\in \sigma(\overrightarrow{\mathrm{k}}) }   (z_{i_1}  z_{i_2} \cdots  z_{i_{\ell-1}}  z_{i_{\ell}}   z_{i_{\ell+1}} \cdots z_{i_k})_{R} \\
  &&\quad\quad\quad -\sum_{(i_1,i_2, \ldots, i_k)\in \sigma(\overrightarrow{\mathrm{k}}) } ( z_{i_\ell}z_{i_1}  z_{i_2} \cdots  z_{i_{\ell-1}}  z_{i_{\ell+1}} \cdots z_{i_k})_{R} \Big)
  \\
  &=&0,
\end{eqnarray*}
where $\sigma(\overrightarrow{\mathrm{k}})$ is  as in Definition \ref{definition-Fueter} and the forth equality follows from the fact
$$ \big( z_{i_1}  z_{i_2} \cdots  z_{i_{\ell-1}} x_{i_\ell}   z_{i_{\ell+1}} \cdots z_{i_k} \big)_R
=  \big( x_{i_\ell}   z_{i_1}  z_{i_2} \cdots  z_{i_{\ell-1}}   z_{i_{\ell+1}} \cdots z_{i_k} \big)_R, \ x_{i_\ell}\in \mathbb{R}. $$
Hence, we get the conclusion $\mathcal{P}_{\mathrm{k}} \in \mathcal{ M}^{L}(M, \mathbb{A})$. Similarly, we can prove the conclusion $\mathcal{P}_{\mathrm{k}} \in \mathcal{ M}^{R}(M, \mathbb{A})$ in which we should consider  the   multiplication in  $\mathcal{P}_{\mathrm{k}}$ from  right to left.  The proof is complete.
\end{proof}

In order to show that $\mathcal{P}_{\mathrm{k}}(x)\in M$ for all $x\in M$, we resort to  a  version of Cauchy-Kovalevskaya extension  starting from some real analytic functions defined in a domain in $\mathbb{R}^{m}$.  For simplicity,  we consider  only the case of polynomials defined in $\mathbb R^{m}$, which is enough for our purpose.
\begin{definition}[CK-extension]\label{Cauchy-Kovalevska-extension}
Let $f_{0}:\mathbb{R}^{m} \to \mathbb{A}$ be a  polynomial. Define the  left  Cauchy-Kovalevskaya extension (CK-extension, for short)  $CK[f_{0}]:  M\to \mathbb{A}$ by
\begin{eqnarray}\label{series-CK-left}
 CK[f_{0}]( x)=CK^{L}[f_{0}]( x) ={\rm exp} (-x_0 \underline{\overline{\partial}}_{\mathcal{B}}) f_{0}(\underline{x})
=\sum_{k=0}^{+\infty} \frac{(-x_0)^{k}}{k!} \big( ( \underline{\overline{\partial}}_{\mathcal{B}} )^{k}f_{0}(\underline{x})\big),
\end{eqnarray}
where $\underline{\overline{\partial}}_{\mathcal{B}}=  \sum_{s=1}^{m}v_s\partial_{s}$.
Similarly, define the  right  Cauchy-Kovalevskaya extension   $CK^{R}[f_{0}]:  M\to \mathbb{A}$ by
$$CK^{R}[f_{0}]( x) =f_{0}(\underline{x}){\rm exp} (-x_0 \underline{\overline{\partial}}_{\mathcal{B}})
=\sum_{k=0}^{+\infty} \frac{(-x_0)^{k}}{k!}\big(f_{0}(\underline{x})( \underline{\overline{\partial}}_{\mathcal{B}} )^{k}\big).
$$
\end{definition}
It should be pointed that $CK[f_{0}]$ is well-defined. The first issue is that there is no non-associativity from Proposition \ref{artin}  since every term in the series  (\ref{series-CK-left}) has only two elements
$\underline{\overline{\partial}}_{\mathcal{B}}$ and $ f_{0}(\underline{x})$. The second issue  is  that the series  (\ref{series-CK-left}) is indeed  a finite sum when $f_0$ is a polynomial. Furthermore, we can rewrite $CK[f_{0}]$ as
\begin{eqnarray}\label{series-CK-two}
CK[f_{0}](x)=
\sum_{k=0}^{+\infty} \frac{x_0^{2k}}{(2k)!} \big(  \Delta_{m}^{k} f_{0}(\underline{x})\big)
-  \sum_{k=0}^{+\infty} \frac{x_0^{2k+1}}{(2k+1)!}\big( \Delta_{m} ^{k} \underline{\overline{\partial}}_{\mathcal{B}} f_{0}(\underline{x})\big),\end{eqnarray}
where $ \Delta_{m}=  \sum _{s=1}^{m} \partial_{s}^{2}$.

\begin{theorem}\label{CK-slice-monogenic}
Let $f_{0}:  \mathbb{R}^{m} \to \mathbb{A}$ be a  polynomial. Then $CK^{L}[f_0]$ (resp. $CK^{R}[f_0]$)   is the unique extension  of $f_{0}$ to $M$ which is left (resp. right)  monogenic. In particular, if $f_{0}$ is $\mathbb{R}$-valued, then $CK^{L}[f_0]= CK^{R}[f_0]$ take both values in  $ M$.
\end{theorem}

\begin{proof}
Here we only consider the left case. From   Proposition \ref{artin}, we have
\begin{eqnarray*}
  \overline{\partial}_{\mathcal{B}} CK[f_{0}](x)& =&
  \partial_0 \sum_{k=0}^{+\infty} \frac{(-x_0)^{k}}{k!} \big( ( \underline{\overline{\partial}}_{\mathcal{B}} )^{k}f_{0}(\underline{x})\big)
  +\overline{\underline{\partial}}_{\mathcal{B}} \sum_{k=0}^{+\infty}  \frac{(-x_0)^{k}}{k!} \big( ( \underline{\overline{\partial}}_{\mathcal{B}} )^{k}f_{0}(\underline{x})\big)
  \\
 &=& \sum_{k=1}^{+\infty} -\frac{(-x_0)^{k-1}}{(k-1)!} \big( ( \underline{\overline{\partial}}_{\mathcal{B}} )^{k}f_{0}(\underline{x})\big)
 + \sum_{k=0}^{+\infty} \frac{(-x_0)^{k }}{ k !} \big( ( \underline{\overline{\partial}}_{\mathcal{B}} )^{k+1}f_{0}(\underline{x})\big)
 \\
 &=& -\sum_{k=0}^{+\infty}  \frac{(-x_0)^{k }}{k!} \big( ( \underline{\overline{\partial}}_{\mathcal{B}} )^{k+1}f_{0}(\underline{x})\big)
 + \sum_{k=0}^{+\infty} \frac{(-x_0)^{k }}{ k !} \big( ( \underline{\overline{\partial}}_{\mathcal{B}} )^{k+1}f_{0}(\underline{x})\big)
 \\
&=&  0,
\end{eqnarray*}
i.e. $CK[f_{0}]\in \mathcal{M}^{L}(M, \mathbb{A})$.
Furthermore, Theorem \ref{Identity-theorem} gives the uniqueness of extension. Finally,   if $f_{0}$ is $\mathbb{R}$-valued, then $CK^{L}[f_0]= CK^{R}[f_0]$ take both values in  $ M$ from (\ref{series-CK-two}). The proof is
complete.
\end{proof}

\begin{remark}\label{CK-real-special-case}
If   $f_0$ is  $\mathbb{R}$-valued, then $CK[f_{0}] (x)=\sum_{s=0}^m \phi_s(x) v_s$ satisfies
 (\ref{i-j-E-1}).
\end{remark}

For  $\mathrm{k}=(k_1,\ldots,k_{m})\in\mathbb{N}^{m}$ and $\underline{x}=(x_1, \ldots, x_m)\in\mathbb{ R}^{m}$, denote
$\underline{x}^{\mathrm{k}}=x_1^{k_1}\ldots x_m^{k_m}$.

\begin{proposition} \label{Ck-P}
For  $\mathrm{k}\in\mathbb{N}^{m}$, it holds that
$$CK[\underline{x}^{\mathrm{k}}]= \mathrm{k}!\mathcal{P}_{ \mathrm{k}}\in \mathcal{M}^{L}(M, M) \cap \mathcal{M}^{R}(M, M).$$
\end{proposition}
\begin{proof}
By  Proposition  \ref{Fueter-Mono} and Theorem \ref{CK-slice-monogenic}, we have $\mathcal{P}_{ \mathrm{k}} \in \mathcal{M}^{R}(M, \mathbb{A})\cap \mathcal{M}^{R}(M, \mathbb{A})$ and $CK[\underline{x}^{\mathrm{k}}] \in \mathcal{M}^{L}(M, M)\cap \mathcal{M}^{R}(M, M)$.
Observing that $\mathrm{k}!\mathcal{P}_{\mathrm{k}}  ( \underline{x})=\underline{x}^{\mathrm{k}}=CK[\underline{x}^{\mathrm{k}}](0+\underline{x})$,  the identity  theorem in Theorem \ref{Identity-theorem} shows $CK[\underline{x}^{\mathrm{k}}] \equiv \mathrm{k}!\mathcal{P}_{ \mathrm{k}}$ on $M$.   The proof is complete.
\end{proof}

From Remark \ref{CK-real-special-case}, Proposition  \ref{Ck-P},  and  Lemma \ref{E-lemma}, we have
\begin{lemma}\label{pa-monogenic}
For $\mathrm{k}\in \mathbb{N}^{m},$ we have $  \mathcal{P}_{\mathrm{k}}  (x) a  \in \mathcal{M}^{L}(M,  \mathbb{A})$ and  $ a \mathcal{P}_{\mathrm{k}}  (x) \in \mathcal{M}^{R}(M, \mathbb{A})$ for all $a\in \mathbb{A}$.
\end{lemma}

\begin{lemma}\label{lemmakth-E}
Let $k\in \mathbb{N}$ and  $P: M \rightarrow \mathbb{A}$ be a  homogeneous polynomial  of degree $k$.
Then we have
 $$  P\in \mathcal{M}^{L}(M, \mathbb{A}) \Leftrightarrow  P(x)= \sum_{|\mathrm{k}|=k} \mathcal{P}_{\mathrm{k}}  (x) a_{\mathrm{k}},\quad  a_{\mathrm{k}}= \partial_{\mathrm{k}} P(0), \mathrm{k}\in \mathbb{N}^{m}. $$
\end{lemma}
\begin{proof}
By Lemma \ref{pa-monogenic}, the sufficiency  can be obtained directly.  Now let us consider the  necessity. Let the polynomial $P\in \mathcal{M}^{L}(M, \mathbb{A})$ be homogeneous of degree $k$.  We have
 $$\sum_{s=0}^m x_s\partial_{s} P(x)= k P(x),$$
 and
 $$\sum _{s=0}^{m}v_{s} \partial_s P(x)=0,$$
which imply
 $$ k P(x)=\sum_{s=0}^m (x_s- x_0 v_s)\partial_{s} P(x)=\sum_{s=1}^m z_s \partial_{s} P(x).$$
Now we iterate the procedure for polynomials $\partial_{s}P$, $s=1,2,\ldots, m,$ which are monogenic  by Proposition \ref{preserving-monogenic} and are homogeneous of degree $(k-1)$.  Hence,
  \begin{eqnarray*}
P(x)& =&\frac{1}{k}\sum_{s=1}^m z_s (\partial_{s} P)(x)
  \\
 &=&\frac{1}{k(k-1)}  \sum_{s,t=1}^m z_s (z_t (\partial_{s,t} P)(x)) \\
 & &\cdots \\
&=& \frac{1}{k! }  \sum_{i_1,\ldots, i_k=1}^m \Big( z_{i_1}\ldots z_{i_k}\frac{\partial^k}{\partial_{x_{i_1}}\ldots \partial_{x_{i_k}}}P(x)\Big)_{R}.
\end{eqnarray*}
Let $\mathrm{k}=(k_1,\ldots,k_{m})\in \mathbb{N}^{m}$ with $k=|\mathrm{k}|$. Grouping all the derivatives of the form $$\partial_{\mathrm{k}} P(x)=\frac{\partial^{k}  }{\partial_{x_{1}}^{k_1}\cdots\partial_{x_{m}}^{k_m} } P(x),$$ we get
   \begin{eqnarray*}
P(x)& =&\frac{1}{k! }  \sum_{i_1,\ldots, i_k=1}^m \Big(z_{i_1}\ldots z_{i_k}\frac{\partial^k}{\partial_{x_{i_1}}\ldots \partial_{x_{i_k}}}P(x) \Big)_{R}
  \\
 &=&\frac{1}{k! } \sum_{|\mathrm{k}|=k}  \sum_{(i_1,i_2, \ldots, i_k)\in \sigma(\overrightarrow{\mathrm{k}})  }
  (z_{i_1}  z_{i_2} \cdots  z_{i_k}  \partial_{\mathrm{k}} P(x))_{R} \\
 &= &\frac{1}{k! } \sum_{|\mathrm{k}|=k}  \sum_{(i_1,i_2, \ldots, i_k)\in \sigma(\overrightarrow{\mathrm{k}})  }
  (z_{i_1}  z_{i_2} \cdots  z_{i_k}  \partial_{\mathrm{k}} P(x))_{L}\\
&=& \frac{1}{k! } \sum_{|\mathrm{k}|=k} \Big( \sum_{(i_1,i_2, \ldots, i_k)\in  \sigma(\overrightarrow{\mathrm{k}})   }
  (z_{i_1}  z_{i_2} \cdots  z_{i_k})_{L} \Big)\partial_{\mathrm{k}} P(x)\\
&=& \sum_{|\mathrm{k}|=k} \mathcal{P}_{\mathrm{k}}(x)\partial_{\mathrm{k}}P(x)\\
&=& \sum_{|\mathrm{k}|=k} \mathcal{P}_{\mathrm{k}}(x) \partial_{\mathrm{k}} P(0),
\end{eqnarray*}
 where  $\sigma(\overrightarrow{\mathrm{k}})$ denotes  all distinguishable permutations  of $\overrightarrow{\mathrm{k}}$ as in Definition \ref{definition-Fueter}, the third equality follows from Proposition \ref{no-order}, and last equality follows from the fact that $P$ has degree $k$.
  The proof is complete.
\end{proof}


To obtain  the  Taylor series expansion of monogenic functions over real alternative $\ast$-algebras, we formulate  a  technical result, instead of  \cite[Theorem 3.18]{Ghiloni-Stoppato-24-2} which is valid for the  associative  case.
 For $\rho >0$, denote $B(\rho)=\{ x \in M: | x|<\rho\}$.

\begin{lemma} \label{Taylor-lemma-E}
Given $y \in M\setminus \{0\}$, it holds that for all $ x\in B(|y|)$ and  $a\in \mathbb{A}$
$$E_{ y}( x)a=  \sum_{k=0}^{+\infty}\Big( \sum_{|\mathrm{k}|=k} \mathcal{P}_{\mathrm{k}}  ( x) (\mathcal{Q}_{ \mathrm{k}}  ( y)a)\Big), $$
where the series converges normally on  $ B(|y|)$.
\end{lemma}
\begin{proof}
 For $m=1$, the conclusion follows directly from the  holomorphic result.
Let $m>1$. Given $y \in M\setminus \{0\}$, for all $a\in \mathbb{A}$ and $x\in B(|y|)$, we have by Proposition \ref{real}
$$\frac{(y-x)^{c}}{|y-x|^{m+1}} a  =  -\frac{1}{m-1}  \Big(\partial_{y,\mathcal{B}}  \frac{1}{|y-x|^{m-1}} \Big)a=-\frac{1}{m-1}  \partial_{y,\mathcal{B}}   \Big( \frac{a}{|y-x|^{m-1}} \Big).
 $$
Note that    the fundamental formula holds
$$ \frac{1}{|y-x|^{m-1}}=\sum_{k=0}^{\infty} \frac{(-1)^{k}}{k!} \langle x, \nabla_{y}\rangle^{k}   \frac{1}{|y|^{m-1}},$$
where   the series converges normally on  $B(|y|)$ and
$ \langle x, \nabla_{y}\rangle = \sum_{s=0}^{m}x_s\partial_{y_s}$,\\
Hence, we get
$$ \frac{(y-x)^{c}}{|y-x|^{m+1}} a =\sum_{k=0}^{\infty}P_{k}(x, y), $$
where the homogeneous polynomial $P_{k}( \cdot, y)$  of degree $k$  given by
$$P_{k}(x, y)=- \frac{(-1)^{k}}{k! (m-1)} \partial_{y,\mathcal{B}}  \big(\langle x, \nabla_{y} \rangle^{k} \frac{a}{|y|^{m-1}}\big)
=\frac{(-1)^{k}}{k!  }   \langle x, \nabla_{y} \rangle^{k} \frac{y^{c}}{|y|^{m+1}}a$$
 is left monogenic in the variable $x\in B(|y|)$  since the   polynomial $ \langle \cdot, \nabla_{y} \rangle^{k} \frac{y^{c}}{|y|^{m+1}}$ satisfies (\ref{i-j-E-1}) and then by Lemma \ref{E-lemma}
 \begin{eqnarray*}
\overline{\partial}_{x, \mathcal{B}} P_{k}(x, y) & =&\frac{(-1)^{k}}{k!  }   \Big(\overline{\partial}_{x, \mathcal{B}} \big(\langle x, \nabla_{y} \rangle^{k} \frac{y^{c}}{|y|^{m+1}}\big)\Big)a
  \\
 &=&\frac{(-1)^{k}}{(k-1)!  }   \Big(\langle x, \nabla_{y} \rangle^{k-1} \big(\overline{\partial}_{y, \mathcal{B}}\frac{y^{c}}{|y|^{m+1}}\big)\Big)a \\
 &= &0.
\end{eqnarray*}

Observing  that, for  $\mathrm{k}=(k_1,\ldots,k_{m})\in\mathbb{N}^{m}$ with $|\mathrm{k}|=k$,
$$\partial_{x, \mathrm{k}} P_{k}(0,  y)=(-1)^{k}   \partial_{ y, \mathrm{k}}   \frac{y^{c}}{| y|^{m+1}}a=\sigma_{m} \mathcal{Q}_{ \mathrm{k}}  ( y)a,\quad $$
we get from Lemma \ref{lemmakth-E}    the desired conclusion
$$E_{y}(x)a= \frac{1}{\sigma_{m}}\sum_{k=0}^{+\infty} P_{k}(x, y) = \sum_{k=0}^{+\infty}\Big( \sum_{|\mathrm{k}|=k} \mathcal{P}_{\mathrm{k}}  (x) (\mathcal{Q}_{ \mathrm{k}}  (y)a)\Big),$$
where the series converges normally on  $ B(|y|)$.
The proof is complete.
\end{proof}

Now we can establish the  Taylor series expansion,  which subsumes \cite[Theorem 1]{Li-Peng-01}  and  \cite[Theorem 11.3.4]{Brackx}. Note that this  Taylor series  is different from \cite[Theorem 6.16]{Xu-Sabadini-25-o} where an additional associator term appears.
\begin{theorem} \label{Taylor-lemma-T}
Let $\rho>0$ and $ f:B(\rho) \rightarrow \mathbb{A}$ be a   monogenic function.   Then,   for all $x\in B(\rho)$,
$$f( x)=  \sum_{k=0}^{+\infty} \Big( \sum_{|\mathrm{k}|=k} \mathcal{P}_{\mathrm{k}}  ( x)\partial_{ \mathrm{k}}  f(0) \Big) , \quad \mathrm{k}=(k_1,\ldots,k_{m})\in \mathbb{N}^{m},$$
where the series converges normally on  $ B(\rho)$.
\end{theorem}
\begin{proof}
Let $f\in\mathcal {M}(B(\rho), \mathbb{A})$.   By Theorem \ref{Cauchy-slice}, it holds that, for all  $x\in B(r)$ with  $ r<\rho$,
\begin{equation*}
f(x)=\int_{\partial B(r) }  E_{y}(x) (n(y)f(y)) dS(y),
 \end{equation*}
where  $n(y)=y/r$ is the unit exterior normal to $\partial B(r)$ at $y$,
 $dS$  stands  for  the   classical   Lebesgue surface  element  in $\mathbb{R}^{m+1}$.\\
Hence, by Lemma \ref{Taylor-lemma-E} and   Corollary \ref{Cauchy-slice-high}, we get
 \begin{eqnarray*}
 f(x) & =& \int_{\partial B(r)}  \sum_{k=0}^{+\infty}  \Big( \sum_{|\mathrm{k}|=k} \mathcal{P}_{\mathrm{k}}  (x) \big(\mathcal{Q}_{ \mathrm{k}}  (y) (n(y)f(y)) \big)  \Big)   dS(y)
  \\
 &=&\sum_{k=0}^{+\infty} \Big( \sum_{|\mathrm{k}|=k} \mathcal{P}_{\mathrm{k}}  (x)\int_{\partial B(r)} \big(\mathcal{Q}_{ \mathrm{k}}  (y) (n(y)f(y)) \big) dS(y) \Big)\\
 &= &\sum_{k=0}^{+\infty} \Big( \sum_{|\mathrm{k}|=k} \mathcal{P}_{\mathrm{k}}  ( x)\partial_{ \mathrm{k}}  f(0) \Big),
\end{eqnarray*}
where $\mathrm{k}=(k_1,\ldots,k_{m})\in \mathbb{N}^{m}$.
  The proof is complete.
\end{proof}
\par
\noindent
\textbf{Acknowledgements}
 The authors are grateful to the referees for their careful reading of
the paper and valuable suggestions and comments, which improve the quality of this paper significantly.\\
\noindent
\textbf{Declarations}
\\
\textbf{Author contributions}
All authors have contributed equally to all aspects of this manuscript and have reviewed its final draft.
\\
\textbf{Conflict of interest}
There is no financial or non-financial interests that are directly or indirectly related to the work submitted for publication.
\\
\textbf{Data availability}
Data sharing is not applicable to this article as no datasets were generated  during the current study.




\vskip 10mm
\end{document}